\documentclass[12pt]{amsart}
\usepackage{mathrsfs, amssymb, amsmath, upgreek, amsfonts, mathtools, mathcomp}
\usepackage{amsthm}

\usepackage[matrix,arrow,curve]{xy}
\usepackage{hyperref}

\usepackage{extarrows}
\makeatletter
\newcommand*{\da@rightarrow}{\mathchar"0\hexnumber@\symAMSa 4B }
\newcommand*{\da@leftarrow}{\mathchar"0\hexnumber@\symAMSa 4C }
\newcommand*{\xdashrightarrow}[2][]{%
  \mathrel{%
    \mathpalette{\da@xarrow{#1}{#2}{}\da@rightarrow{\,}{}}{}%
  }%
}
\newcommand{\xdashleftarrow}[2][]{%
  \mathrel{%
    \mathpalette{\da@xarrow{#1}{#2}\da@leftarrow{}{}{\,}}{}%
  }%
}
\newcommand*{\da@xarrow}[7]{%
  \sbox0{$\ifx#7\scriptstyle\scriptscriptstyle\else\scriptstyle\fi#5#1#6\m@th$}%
  \sbox2{$\ifx#7\scriptstyle\scriptscriptstyle\else\scriptstyle\fi#5#2#6\m@th$}%
  \sbox4{$#7\dabar@\m@th$}%
  \dimen@=\wd0 %
  \ifdim\wd2 >\dimen@
    \dimen@=\wd2 %
  \fi
  \count@=2 %
  \def\da@bars{\dabar@\dabar@}%
  \@whiledim\count@\wd4<\dimen@\do{%
    \advance\count@\@ne
    \expandafter\def\expandafter\da@bars\expandafter{%
      \da@bars
      \dabar@
    }%
  }%
  \mathrel{#3}%
  \mathrel{%
    \mathop{\da@bars}\limits
    \ifx\\#1\\%
    \else
      _{\copy0}%
    \fi
    \ifx\\#2\\%
    \else
      ^{\copy2}%
    \fi
  }%
  \mathrel{#4}%
}
\makeatother
\newcommand{\xar}{\xrightarrow{\hspace{1.8em}}}
\newcommand{\xarr}[1]{\xrightarrow{\hspace{0.9em}#1\hspace{0.9em}}}

\newcommand{\Ind}{\operatorname{Ind}}
\newcommand{\Exc}{\operatorname{Exc}}
\newcommand{\red}{{\mathrm{red}}}
\newcommand{\Sing}{\operatorname{Sing}}

\newcommand{\Aut}{\operatorname{Aut}}

\newcommand{\kod}{\varkappa} 
\newcommand{\q}{\operatorname{q}} 
\newcommand{\alb}{\operatorname{\upalpha}}
\newcommand{\Alb}{\operatorname{Alb}}

\newcommand{\Bim}{\operatorname{Bim}}

\newcommand{\kf}{\operatorname{\uppi}}
\newcommand{\ad}{\operatorname{a}}

\newcommand{\g}{\operatorname{g}}
\newcommand{\p}{\operatorname{p}_{\mathrm{a}}}
\newcommand{\pg}{\operatorname{p}_{\mathrm{g}}}
\newcommand{\chit}{\chi_{\mathrm{top}}}

\newcommand{\rb}{\operatorname{b}}

\newcommand{\ZZ}{\mathbb{Z}}
\newcommand{\PP}{\mathbb{P}}
\newcommand{\QQ}{\mathbb{Q}}

\newcommand{\OOO}{\mathscr{O}}

\newcommand{\comp}{\mathbin{\scriptstyle{\circ}}}

\def \ge {\geqslant}
\def \le {\leqslant}

\makeatletter
\@addtoreset{equation}{section}
\makeatother

\theoremstyle{plain}
\newtheorem{theorem}[subsection]{Theorem}

\newtheorem{lemma}[subsection]{Lemma}

\newtheorem{proposition}[subsection]{Proposition}

\newtheorem{corollary}[subsection]{Corollary}

\theoremstyle{definition}
\newtheorem{definition}[subsection]{Definition}
\newtheorem*{definition*}{Definition}

\newtheorem{remark}[subsection]{Remark}

\newcommand{\xref}[1]{{\rm \ref{#1}}}

\title{Finite groups of bimeromorphic selfmaps of non-uniruled K\"ahler threefolds}

\author{Yu.~G.~Prokhorov}
\author{C.~A.~Shramov}

 \address{\emph{Yuri Prokhorov}
 \newline
 \textnormal{Steklov Mathematical Institute of RAS,
 8 Gubkina street, Moscow 119991, Russia.
 }
 \newline
 \textnormal{
 HSE University, Russian Federation,
 Laboratory of Algebraic Geometry, 6 Usacheva str., Moscow, 119048, Russia.
 }
 \newline
 \textnormal{\texttt{prokhoro@mi-ras.ru}}}

 \address{\emph{Constantin Shramov}
 \newline
 \textnormal{Steklov Mathematical Institute of RAS,
 8 Gubkina street, Moscow 119991, Russia.
 }
 \newline
 \textnormal{
 HSE University, Russian Federation,
 Laboratory of Algebraic Geometry, 6 Usacheva str., Moscow, 119048, Russia.
 }
 \newline
 \textnormal{\texttt{costya.shramov@gmail.com}}}

 \thanks{This work is supported by the Russian Science Foundation under grant \textnumero 18-11-00121.}
 \date{}

\begin{document}

\begin{abstract}
We prove the Jordan property for groups of bimeromorphic selfmaps of three-dimensional compact K\"ahler varieties of non-negative Kodaira dimension
and positive irregularity.
\end{abstract}

\maketitle
\tableofcontents

\section{Introduction}

In the theory of automorphism groups of complex varieties, the following notion plays an important
role.

\begin{definition}[{\cite[Definition~2.1]{Popov2011}}]
\label{definition:Jordan}
A group~$\Gamma$ is called \emph{Jordan}
(alternatively, one says that~$\Gamma$ has \emph{Jordan property}),
if there exists a constant $J=J(\Gamma)$ such that any finite subgroup~\mbox{$G\subset\Gamma$}
contains a normal abelian subgroup $A\subset G$ of index at most~$J$.
\end{definition}

It appears that Jordan property holds for many groups of geometric origin, including automorphism groups and groups of bimeromorphic selfmaps
of many compact complex varieties.
For instance, if $X$ is a compact complex surface, then the automorphism group of $X$ is always Jordan (see \cite[Theorem~1.6]{Prokhorov-Shramov-CCS}),
while its group of bimeromorphic selfmaps is Jordan if and only if $X$ is not bimeromorphic to a product of $\PP^1$ and an elliptic curve
(see \cite[Theorem~1.7]{Prokhorov-Shramov-CCS}).
As for higher dimensions, much less is understood.
For instance, it is unclear whether the automorphism group of an arbitrary compact complex variety is Jordan or not;
a positive answer is known only for the neutral component of such a group, see \cite[Theorem~5]{Popov-Lie}.
However, there are several important results in the case of K\"ahler varieties.
To start with, the following theorem was proved in~\cite{Kim}.

\begin{theorem}
\label{theorem:Kim}
Let $X$ be a compact K\"ahler variety.
Then the automorphism group of~$X$ is Jordan.
\end{theorem}

Theorem~\ref{theorem:Kim} was recently generalized to the case of automorphism groups of compact complex spaces of Fujiki's class~$\mathcal{C}$ in~\cite{MPZ}
(cf. \cite{ProkhorovShramov-Moishezon}).
In~\cite{Prokhorov-Shramov-2020}, groups of bimeromorphic selfmaps of
uniruled compact K\"ahler varieties of dimension~$3$ were studied
from the point of view of Jordan property. In higher dimensions, partial results for $\PP^1$-bundles over appropriate bases
were obtained in \cite{BandmanZarhin-Bim} and~\cite{BandmanZarhin-BimII}
(cf.~\cite{BandmanZarhin-PoorTori}). In particular, \cite{Zarhin-ToriTheta} provides plenty of examples of uniruled
compact K\"ahler varieties with non-Jordan groups of bimeromorphic selfmaps.

Denote by~$\kod(X)$
the Kodaira dimension of a compact complex variety
$X$, and by~$\q(X)$ its irregularity; we refer the reader to Section~\ref{sect:prelim} for details.
Recall that a three-dimensional compact K\"ahler variety is uniruled if and only if its Kodaira dimension is negative, see \cite[Corollary~1.4]{Hoering-Peternell-2016}:
this is a very deep fact which follows from
the results of~\cite{Brunella}, \cite{DemaillyPeternell}, and \cite{Hoering-Peternell-2016}.
The purpose of this paper is to prove the following theorem concerning the groups of bimeromorphic selfmaps of non-uniruled three-dimensional compact K\"ahler varieties;
we refer the reader to Section~\ref{sect:prelim} for notation and conventions.

\begin{theorem}
\label{theorem:main}
Let $X$ be a three-dimensional compact K\"ahler variety with~\mbox{$\kod(X)\ge 0$} and~\mbox{$\q(X)>0$}.
Then the group of bimeromorphic selfmaps of~$X$ is Jordan.
\end{theorem}

The plan of the paper is as follows.
In Section~\ref{sect:prelim}
we collect basic definitions and auxiliary facts concerning compact complex
varieties and manifolds.
In Section~\ref{section:equivariant}
we study groups of bimeromorphic selfmaps that preserve
fibrations on compact complex varieties. In Section~\ref{sect:pseudo}
we consider groups of pseudoautomorphisms of compact K\"ahler manifolds.
In particular, we show in Proposition~\ref{proposition:pseudo} that the group of pseudoautomorphisms of a compact K\"ahler manifold is always Jordan;
a projective analog of this result is well known to experts.
In Section~\ref{section:Albanese}
we collect information about Albanese maps of compact complex varieties.
In Section~\ref{section:indeterminacy}
we discuss some technical facts on the indeterminacy loci of pseudoutomorphisms.
In Section~\ref{section:more} we make a couple of additional observations on compact complex surfaces and their automorphism groups.
Finally, in Section~\ref{section:proof}
we complete the proof of Theorem~\xref{theorem:main}.

For a non-uniruled algebraic variety $X$
of arbitrary dimension a much stronger assertion
than Theorem~\ref{theorem:main}
is valid: the group of birational selfmaps of~$X$ is Jordan without any restriction on the irregularity,
see~\mbox{\cite[Theorem~1.8(ii)]{Prokhorov-Shramov-Bir}}.
The proof of this result in \cite{Prokhorov-Shramov-Bir} uses the equivariant Minimal Model Program
(see \cite{P:G-MMP} for a brief introduction).
The proofs in this paper also use the MMP for three-dimensional compact K\"ahler varieties
which is available due to~\cite{Hoering-Peternell-2016}.
We expect that an analog of Theorem~\ref{theorem:main}  is valid in arbitrary dimension
and can be proved independently of the~MMP.

Furthermore, our proof of Jordan property for pseudoautomorphism
groups is based on Fujiki's result
from~\cite{Fujiki:bimeromorphic} (see Theorem~\ref{theorem:Fujiki-pseudo} below)
which is known  only for smooth K\"ahler varieties. It is desirable to generalize it for K\"ahler varieties with terminal singularities.
This would allow to largely simplify our arguments, and to remove the assumption that~\mbox{$\q(X)>0$} in Theorem~\xref{theorem:main}.\footnote{After this 
paper was written, A.\,Golota in~\cite{Golota} generalized 
the result of Fujiki to the case of singular K\"ahler varieties and proved a stronger 
version of Theorem~\xref{theorem:main}.}

\medskip
We are grateful to Ch.\,Hacon who pointed out a gap in the first draft of the paper.
We also thank the anonymous referee for helpful suggestions.

\section{Preliminaries}
\label{sect:prelim}

In this section we collect some auxiliary facts about compact complex
varieties and manifolds. We refer the reader to \cite{Ueno1975}
for the most basic facts and definitions.
In particular, by a \textit{complex variety} we mean an
irreducible reduced complex space. A \emph{morphism} of complex varieties is a holomorphic map between them.
A smooth complex variety is called a \textit{complex manifold}.
A \emph{complex surface} is a complex
manifold of dimension~$2$.

Given a compact complex variety $X$, by $\ad(X)$ we denote the algebraic dimension of $X$,
i.e. the transcendence degree of the field of meromorphic functions on~$X$. By $\q(X)$
we denote the irregularity of $X$, i.e. the dimension of  $H^1(X',\OOO_{X'})$,
where $X'$ is an arbitrary compact complex manifold bimeromorphic to $X$.
Similarly, the Kodaira dimension  $\kod(X)$ is defined as the
Kodaira dimension of its smooth compact model $X'$,
see~\mbox{\cite[Definition~6.5]{Ueno1975}}.
By $\Bim(X)$ we denote the group of biholomorphic selfmaps of $X$.
A \emph{Zariski open subset} of $X$ is a
subset of the form~\mbox{$X\setminus\Sigma$}, where $\Sigma$ is a closed
(analytic) subset in~$X$.
A \emph{typical point} of~$X$ is a point
of some non-empty Zariski open subset of~$X$;
a \emph{typical fiber}
of a (meromorphic) map~\mbox{$\phi\colon X\dasharrow Y$}
is a fiber over a typical point of~$Y$.
For a meromorphic map $\chi\colon X\dasharrow Y$,
we denote by $\Ind(\chi)$ the \emph{indeterminacy locus}
of $\chi$, i.e. the minimal closed analytic subset $V\subset X$ such that the restriction of $\chi$ to $X\setminus V$
is holomorphic.

Recall that the canonical class of a normal  complex variety $X$ is the reflexive rank one
sheaf
$$
\upomega_X=j_* \Omega_{X_0}^{\dim(X)},
$$
where $X_0\subset X$ is the smooth locus and
$j\colon X_0 \hookrightarrow X$ is the embedding,
see \cite[\S1A]{Reid:can-3-folds}.
Note that  in contrast with the projective case the canonical class~$\upomega_X$ need not be represented by a Weil divisor.
However,  $\upomega_X$ is always represented by a Weil divisor locally,
in  a small analytic neighborhood of any point. Thus,
sometimes we will abuse notation and
write $K_X$ instead of $\upomega_X$.
Note also that throughout this paper we consider
varieties of non-negative Kodaira dimension; for such a variety~$X$ the reflexive sheaf
$$
\upomega_X^{[n]}=\big(\upomega_X^{\otimes n}\big)^{\vee\vee}
$$
is represented by a Weil divisor $nK_X$ for some positive integer~$n$.
Indeed, for such a divisor one can take~\mbox{$nK_X=f_* D$}, where~\mbox{$f\colon\tilde X\to X$} is a resolution of singularities, and $D$ is an effective divisor from non-empty linear system~\mbox{$|nK_{\tilde X}|$}.

For the  basic terminology, definitions, and facts concerning the K\"ahler Minimal Model Program
we refer to \cite{Hoering-Peternell-2016}.
We emphasize only the following distinctions.
Since the canonical class~$\upomega_X$ need not be represented by a Weil divisor,
in the case of non-algebraic complex varieties
the definition of a normal
$\QQ$-factorial singularity includes the requirement that
the sheaf $\upomega_X^{[n]}$ is invertible for some $n$.
We remind the reader that a terminal singularity is a normal $\QQ$-Gorenstein singularity of a complex variety
such that all of its discrepancies are positive.
In particular, if a variety $X$ has only terminal (or just normal $\QQ$-Gorenstein) singularities,
the intersection numbers of $\upomega_X$ with all the curves on $X$ are well defined.
A canonical singularity is a normal $\QQ$-Gorenstein singularity of a complex variety
such that all of its discrepancies are non-negative.
All canonical (and in particular terminal) singularities are rational, see e.g. \cite[Theorem~5.22]{KollarMori}.

One says that a group $\Gamma$ has  \emph{bounded finite subgroups}
if there exists a constant $B=B(\Gamma)$
such that every finite subgroup of $\Gamma$
has order at most $B$.

\begin{lemma}
\label{lemma:group-theory}
Let
$$
1\xar\Gamma'\xar\Gamma\xar\Gamma''
$$
be an exact sequence of groups.
Suppose that the group $\Gamma''$ has bounded finite subgroups.
Then the group $\Gamma$ is Jordan
if and only if the group $\Gamma'$ is Jordan.
\end{lemma}
\begin{proof}
Obvious.
\end{proof}

A group $\Gamma$ is called \emph{strongly Jordan} if it is Jordan and there exists a constant $r=r(\Gamma)$ such that
every finite subgroup of~$\Gamma$ is generated by at most~$r$ elements.

\begin{lemma}[{\cite[Lemma~2.8]{Prokhorov-Shramov-Bir}}]
\label{lemma:group-theory-2}
Let
$$
1\xar\Gamma'\xar\Gamma\xar\Gamma''
$$
be an exact sequence of groups.
Suppose that the group $\Gamma''$ is strongly Jordan and the group~$\Gamma'$ has bounded finite
subgroups.
Then the group $\Gamma$ is Jordan.
\end{lemma}

The following classical theorem was proved by H.\,Minkowski.

\begin{theorem}[{see e.g.~\cite[Theorem~1]{Serre2007}}]
\label{theorem:Minkowski}
For every positive integer
$n$ the group~{$\mathrm{GL}_n(\mathbb{Z})$} has bounded finite subgroups.
\end{theorem}

\begin{corollary}
\label{corollary:Minkowski}
Let $\Lambda$ be a finitely generated abelian group.
Then the group $\Aut(\Lambda)$
has bounded finite subgroups.
\end{corollary}

The following assertion is well known, see e.g. \cite[Proposition~1.2.1]{BL} or \cite[Theorem~8.4]{Prokhorov-Shramov-CCS}.

\begin{theorem}\label{theorem:torus-Aut}
Let $T$ be a complex torus of dimension $n$. Then
$$
\Aut(T)\cong T\rtimes\Gamma,
$$
where $\Gamma$ is a subgroup of $\mathrm{GL}_{2n}(\mathbb{Z})$.
\end{theorem}

The next result follows from
\cite[Lemma~9.11]{Ueno1975} and \cite[Lemma~3.5]{Graf2018}.

\begin{proposition}
\label{proposition:torus-holomorphic}
Let $X$ be a compact complex variety with rational singularities,
and let $T$ be a complex torus.
Let $\zeta\colon X\dashrightarrow T$ be a meromorphic map. Then $\zeta$ is holomorphic.
\end{proposition}

\begin{corollary}
\label{corollary:torus}
Let $T$ be a complex torus of dimension $n$.
Then the group $\Bim(T)$ is Jordan.
Moreover, there exists a
positive integer $r=r(n)$ which depends only on~$n$ but not on~$T$ such that every finite subgroup
of $\Bim(T)$ is generated by at most~$r$ elements. In particular, $\Bim(T)$ is strongly Jordan.
\end{corollary}

\begin{proof}
One has $\Bim(T)=\Aut(T)$ by Proposition~\ref{proposition:torus-holomorphic}.
Thus the first assertion follows from
Theorems~\ref{theorem:torus-Aut}
and~\ref{theorem:Minkowski} together with Lemma~\ref{lemma:group-theory}
(alternatively, it can be obtained from
Theorem~\ref{theorem:Kim}). The second assertion follows directly from
Theorems~\ref{theorem:torus-Aut}
and~\ref{theorem:Minkowski}.
\end{proof}

\begin{theorem}[{see \cite[Corollary~14.3]{Ueno1975}}]
\label{theorem:general-type}
Let $X$ be a compact complex variety with~\mbox{$\dim(X)=\kod(X)$}.
Then the group $\Bim(X)$ is finite.
\end{theorem}

\begin{lemma}\label{lemma:fiber-kappa}
Let $X$ and $Y$ be  compact complex varieties,
and let $\phi\colon X\dashrightarrow Y$
be a dominant meromorphic map.
Suppose that $\kod(X)\ge 0$.
Let $F$ be a typical fiber
of $\phi$, and let $F'$ be an irreducible component of $F$.
Then $\kod(F')\ge 0$.
\end{lemma}

\begin{proof}
Applying the resolution of singularities and indeterminacies, we may assume that
$X$ and $Y$ are smooth, and the map
$\phi$ is holomorphic. In particular,
this means that $F$ is a compact
complex manifold, and~$F'$ is a connected component of $F$.
By adjunction one has
$$
\upomega_{F'}\cong\upomega_X\vert_{F'}.
$$
Therefore, since $\upomega_X^{\otimes n}$ is represented by an effective divisor for some positive integer~$n$,
the same holds for~$\upomega_{F'}$. Thus,
we have $\kod(F')\ge 0$.
\end{proof}

For most of the compact complex surfaces, their groups of bimeromorphic selfmaps are Jordan.
More precisely, the following is known.

\begin{theorem}[{see \cite[Theorem~1.7]{Prokhorov-Shramov-CCS}}]
\label{theorem:surface}
Let $X$ be a compact complex surface
with $\kod(X)\ge 0$. Then
the group $\Bim(X)$ is strongly Jordan.
\end{theorem}

\begin{remark}\label{remark:curve}
Let $X$ be a (smooth connected) compact complex curve.
It is easy to see that the group $\Bim(X)=\Aut(X)$ is strongly Jordan.
\end{remark}

If a group acts on a compact K\"ahler manifold of non-negative Kodaira dimension with a fixed point, then
it has bounded finite subgroups.

\begin{theorem}[{\cite[Theorem~1.5]{Prokhorov-Shramov-BFS}}]
\label{theorem:fixed-point-BFS}
Let $X$ be a compact K\"ahler manifold of non-negative Kodaira dimension, and let $P$
be a point on $X$. Then the stabilizer of $P$ in $\Aut(X)$
has bounded finite subgroups.
\end{theorem}

\section{Equivariant fibrations}
\label{section:equivariant}

In this section we make several observations about groups of bimeromorphic selfmaps that preserve
fibrations on complex varieties.

Given a dominant meromorphic map $\alpha\colon X\dashrightarrow Y$ of
compact complex varieties,
we denote by~\mbox{$\Bim(X; \alpha)$} the subgroup
in $\Bim(X)$ that consists of all bimeromorphic selfmaps of $X$
which map the fibers of $\alpha$ again to the fibers of $\alpha$.
In other words, if $\chi\in\Bim(X;\alpha)$, then
for every two points $P,\, Q\in X\setminus\Ind(\chi)$ such that $\alpha(P)=\alpha(Q)$ we have $\alpha(\chi(P))=\alpha(\chi(Q))$. There is a natural
homomorphism
$$
h_\alpha\colon \Bim(X;\alpha)\xar\Bim(Y),
$$
and the map $\alpha$ is equivariant with respect to $\Bim(X;\alpha)$.
Denote by $\Bim(X)_\alpha$ the kernel of the homomorphism~$h_\alpha$.

The proof of the following lemma is similar to that of \cite[Lemma~4.1]{Prokhorov-Shramov-2020}.

\begin{lemma}
\label{lemma:fiberwise-embedding}
Let $X$ and $Y$ be compact complex varieties, and let~\mbox{$\alpha\colon X\dasharrow Y$} be a dominant meromorphic map.
Then there is a constant $I=I(\alpha)$ with the following property.
Let~\mbox{$G_i, i\in\mathbb{N}$}, be a countable family of finite subgroups in~\mbox{$\Bim(X)_{\alpha}$}.
Then there exists a reduced
fiber $F$ of the map $\alpha$, and its irreducible component $F'$ of dimension $\dim(X)-\dim(Y)$, such that in every group~$G_i$ there is a subgroup of index at most $I$ which is isomorphic to a subgroup of~\mbox{$\Bim(F')$}.
Moreover, if~\mbox{$\dim(Y)>0$}, and we are
given a countable union $\Xi$ of proper closed analytic subsets in $Y$, then the fiber $F$ can be chosen so that the point $\alpha(F)$ does not lie in~$\Xi$.
\end{lemma}

\begin{proof}
Let $\Delta\subset Y$ be the minimal closed subset of $Y$ such that every point $P$ of $Y\setminus\Delta$ is smooth,
the fiber $\alpha^{-1}(P)$ is reduced, and every irreducible component of $\alpha^{-1}(P)$
has dimension~\mbox{$\dim(X)-\dim(Y)$}.
Then $Y\setminus \Delta$ is a dense open subset of~$Y$. Enlarging $\Delta$ if necessary, we may assume that the fibers over all the points
of $Y\setminus\Delta$
have the same number $N$ of irreducible components.
Denote
$$
\mathcal{G}=\bigcup_i G_i;
$$
thus, $\mathcal{G}$ is a countable set of elements in $\Bim(X)_\alpha$.

Let $\gamma$ be an element of the group~\mbox{$\Bim(X)_{\alpha}$}.
Consider the set~\mbox{$\nabla_\gamma\subset Y$} consisting of all the points~$P$ for which
$\Ind(\gamma)$ contains an irreducible component of the fiber~$\alpha^{-1}(P)$. Thus the map~$\gamma$ is defined in a typical point of every irreducible component of the fiber~$\alpha^{-1}(P)$ over every point~\mbox{$P\in Y\setminus \nabla_\gamma$}. Moreover, for any point
$$
P\in Y\setminus\big(\Delta\cup\nabla_\gamma\cup\nabla_{\gamma^{-1}}\big)
$$
the restriction $\gamma\vert_{F'}$ of the map $\gamma$
to every irreducible component $F'$ of the fiber $\alpha^{-1}(P)$ is a bimeromorphic map of
$F'$ to its image $\gamma(F')$, and the image $\gamma(F')$ does not coincide with an image of any other
irreducible component of~$\alpha^{-1}(P)$.

Consider the subset $D_\gamma\subset Y\setminus (\Delta\cup\nabla_\gamma)$ consisting of all the points~$P$
such that for a typical point~$Q$ of some irreducible component of the fiber $\alpha^{-1}(P)$ one has $\gamma(Q)=Q$.
Let $\overline{D_\gamma}$ be the closure of $D_\gamma$ in $Y$.
Then for any point
$$
P\in Y\setminus\big(\Delta\cup\nabla_\gamma\cup\nabla_{\gamma^{-1}}\cup\overline{D_\gamma}\big)
$$
the restriction $\gamma\vert_{F'}$ of the map $\gamma$
to every irreducible component $F'$ of the fiber $\alpha^{-1}(P)$ is not the identity map of $F'$, provided that $\gamma$ itself is not the identity map of~$X$.

The sets $\Delta$, $\nabla_\gamma$, and $\overline{D_\gamma}$ are proper closed analytic subsets in $Y$. If~\mbox{$\dim(Y)>0$}, fix also a subset~$\Xi$ which is a countable union of proper closed analytic subsets in $Y$.
Since the field~$\mathbb{C}$ is uncountable,  $Y$ cannot
be represented as a countable union of proper closed subsets. Hence
the complement
$$
U=Y\setminus\left(\Xi\cup\Delta\cup\bigcup_{\gamma\in\mathcal{G}\setminus\{\mathrm{id}\}}\left(\nabla_\gamma\cup \overline{D_\gamma}\right)\right)
$$
is non-empty.

Let $P$ be a point of $U$, and let $F$ be the fiber of $\alpha$ over $P$. Every element $\gamma\in \mathcal{G}$ defines a permutation of
the set of $N$ irreducible components of $F$. Thus, for every $i$ we have a homomorphism $G_i\to\mathfrak{S}_N$ to the symmetric group
of degree~$N$. Denote by $K_i\subset G_i$ the kernel of this homomorphism. Then the index of $K_i$ in $G_i$ is at most
$N!=|\mathfrak{S}_N|$. Moreover, every element of $K_i$ maps every irreducible component $F'$ of $F$ to itself, and every
non-trivial element of $K_i$ restricts to a non-trivial bimeromorphic selfmap of $F'$. Therefore, all the groups $K_i$ are embedded
into the group~$\Bim(F')$.
\end{proof}

\begin{lemma}
\label{lemma:bounded-base}
Let $X$ and $Y$ be compact complex varieties, and let
$\alpha\colon X\dasharrow Y$ be a dominant meromorphic map.
Let $F$ be a typical fiber of $\alpha$.
Suppose that for any irreducible component
$F'$ of~$F$ the group $\Bim(F')$ is Jordan. Suppose also that the image
$h_\alpha(\Bim(X;\alpha))\subset\Bim(Y)$ has bounded finite subgroups.
Then the group $\Bim(X;\alpha)$ is Jordan.
In particular, if under these assumptions the map $\alpha$ is equivariant with respect
to the whole group $\Bim(X)$, then $\Bim(X)$ is Jordan.
\end{lemma}

\begin{proof}
Suppose that the group $\Bim(X;\alpha)$ is not Jordan.
By Lemma~\ref{lemma:group-theory}, this means that the group $\Bim(X)_\alpha$ is also not Jordan.
Hence the group $\Bim(X)_\alpha$ contains a countable family of subgroups~\mbox{$G_i$, $i\in\mathbb{N}$}, such
that the minimal indices $J_i$ of normal abelian subgroups of $G_i$ form an unbounded
sequence. On the other hand, by Lemma~\ref{lemma:fiberwise-embedding} there exists a constant $I$, a
typical fiber~$F$ of the map $\alpha$, and an irreducible component $F'$ of $F$,
such that every~$G_i$ contains a
subgroup~$K_i$ of index at most $I$ which can be embedded into the group~\mbox{$\Bim(F')$}.
Since $\Bim(F')$ is a Jordan group, we conclude that the minimal index of an abelian subgroup of $K_i$ is bounded
by a constant~$J$ independent of $i$. Therefore, the minimal index of an abelian subgroup of $G_i$ is bounded
by the constant~$IJ$, and hence the minimal index of a normal abelian subgroup of $G_i$ is also bounded.
This contradicts the unboundedness of the indices~$J_i$.
\end{proof}

\begin{corollary}
\label{corollary:3-fold-over-surface}
Let $X$ be a three-dimensional compact complex variety, let~$Z$ be a compact complex surface,
and let~\mbox{$\alpha\colon X\dasharrow Z$} be a dominant meromorphic map.
Suppose that the image~\mbox{$h_\alpha(\Bim(X;\alpha))\subset\Aut(Z)$}
has bounded finite subgroups.
Then the group $\Bim(X;\alpha)$ is Jordan.
In particular, if under these assumptions the map $\alpha$ is equivariant with respect
to the whole group $\Bim(X)$, then~\mbox{$\Bim(X)$} is Jordan.
\end{corollary}

\begin{proof}
Let $F$ be a typical fiber
of $\alpha$, and let $F'$ be an irreducible component of $F$.
Then $F'$ is a smooth projective curve.
Hence $\Bim(F')=\Aut(F')$ is Jordan, see Remark~\ref{remark:curve}.
Therefore, the required assertion follows from Lemma~\ref{lemma:bounded-base}.
\end{proof}

\begin{corollary}
\label{corollary:3-fold-over-curve}
Let $X$ be a three-dimensional compact complex variety with $\kod(X)\ge 0$, let~$B$ be a curve,
and let~\mbox{$\alpha\colon X\dasharrow B$} be a dominant meromorphic map.
Suppose that the image~\mbox{$h_\alpha(\Bim(X;\alpha))\subset\Aut(B)$}
has bounded finite subgroups.
Then the group $\Bim(X;\alpha)$ is Jordan.
In particular, if under these assumptions the map $\alpha$ is equivariant with respect
to the whole group $\Bim(X)$, then~\mbox{$\Bim(X)$} is Jordan.
\end{corollary}

\begin{proof}
Let $F$ be a typical fiber
of $\alpha$, and let $F'$ be an irreducible component of $F$.
Then $\kod(F')\ge 0$ by Lemma~\ref{lemma:fiber-kappa}.
Hence $\Bim(F')$ is Jordan by Theorem~\ref{theorem:surface}.
Therefore, the required assertion follows from Lemma~\ref{lemma:bounded-base}.
\end{proof}

\begin{corollary}\label{corollary:3-fold-over-elliptic-curve-fixed-point}
Let $X$ be a three-dimensional compact complex variety with $\kod(X)\ge 0$, let $A$ be a smooth projective curve,
and let~\mbox{$\alpha\colon X\dasharrow A$} be a dominant meromorphic map.
Suppose that either the genus of $A$ is at least $2$, or $A$ is elliptic and
the image~\mbox{$h_\alpha(\Bim(X;\alpha))\subset\Aut(A)$} preserves a non-empty finite subset of $A$.
Then the group $\Bim(X;\alpha)$ is Jordan.
In particular, if under these assumptions the map $\alpha$ is equivariant with respect
to the whole group $\Bim(X)$, then~\mbox{$\Bim(X)$} is Jordan.
\end{corollary}

\begin{proof}
We see from the assumptions that the group $h_\alpha(\Bim(X;\alpha))$ is finite.
Therefore, the assertion follows from Corollary~\ref{corollary:3-fold-over-curve}.
\end{proof}

\begin{corollary}\label{corollary:over-torus}
Let $X$ be a three-dimensional compact K\"ahler variety with $\kod(X)\ge 0$, and
let~\mbox{$\alpha\colon X  \dashrightarrow T$}  be a
dominant  meromorphic map
to a two-dimensional complex torus $T$.
Suppose that there exists a subvariety $V\subsetneq T$
invariant with respect to~\mbox{$h_\alpha(\Bim(X;\alpha))$}.
Then the group $\Bim(X;\alpha)$ is Jordan.
In particular, if under these assumptions the map $\alpha$ is equivariant with respect
to the whole group $\Bim(X)$, then~\mbox{$\Bim(X)$} is Jordan.
\end{corollary}

\begin{proof}
Recall that the action of~\mbox{$h_\alpha(\Bim(X;\alpha))$} on $T$ is regular by Proposition~\ref{proposition:torus-holomorphic}.
If the group~\mbox{$h_\alpha(\Bim(X;\alpha))$} has bounded finite subgroups, then
$\Bim(X;\alpha)$ is Jordan by Corollary~\ref{corollary:3-fold-over-surface}.

Let $V_1$ be an irreducible component of $V$, and let $\Gamma_1\subset h_\alpha(\Bim(X;\alpha))$
be the stabilizer of~$V_1$. Then $\Gamma_1$ is a finite index subgroup in $h_\alpha(\Bim(X;\alpha))$.
If $\Gamma_1$ acts on $V_1$ with a fixed point, then it has bounded finite subgroups by Theorem~\ref{theorem:fixed-point-BFS}
(or by Theorems~\ref{theorem:torus-Aut}
and~\ref{theorem:Minkowski}). Hence $h_\alpha(\Bim(X;\alpha))$ has bounded finite subgroups  as well.
In particular, this applies to the case when $V_1$ is a point itself.

Suppose that $V_1$ is a curve; then $V_1$ is not rational.
If $V_1$ is singular, then a subgroup of finite index in $\Gamma_1$ acts on
$T$ with a fixed point, and so $\Gamma_1$ and $h_\alpha(\Bim(X;\alpha))$ have bounded finite subgroups.
Thus, we assume that $V_1$ is smooth.
If $\g(V_1)>1$, then the group $\Aut(V_1)$
is finite. This means that a subgroup of finite index in $\Gamma_1$ acts on $V_1$ trivially, and in particular
has a fixed point. As before, we conclude that $\Gamma_1$ and $h_\alpha(\Bim(X;\alpha))$ have bounded finite subgroups in this case.

Therefore, we may assume that $V_1$ is an elliptic curve. Consider the quotient
torus~\mbox{$T_1=T/V_1$}, where the group structure on $T$ is chosen in such a way
that $V_1$ is a subgroup of $T$ (i.e., the neutral element of $T$ is contained in $V_1$).
Note that $T_1$ and the quotient map $T\to T_1$ do not depend on this choice, and hence
the map $T\to T_1$ is $\Gamma_1$-equivariant.
Thus the composition $X \to T\to T_1$ is
equivariant with respect to
the preimage $\tilde{\Gamma}_1$ of $\Gamma_1$ in $\Bim(X;\alpha)$. Moreover,
the image of $\tilde{\Gamma}_1$ (or of~$\Gamma_1$) in $\Aut(T_1)$ preserves a point on the elliptic curve $T_1$.
Therefore, the group $\tilde{\Gamma}_1$ is Jordan by Corollary~\ref{corollary:3-fold-over-elliptic-curve-fixed-point}.
Since $\tilde{\Gamma}_1$ has finite index in $\Bim(X;\alpha)$, the latter group is Jordan as well.
\end{proof}

\begin{lemma}
\label{lemma:bounded-fiber}
Let $X$ and $Y$ be compact complex varieties, and let
$\alpha\colon X\dasharrow Y$ be a dominant meromorphic map.
Let $F$ be a typical fiber of $\alpha$.
Suppose that for any irreducible component
$F'$ of~$F$ the group $\Bim(F')$ has bounded finite subgroups.
Suppose also that the image~\mbox{$h_\alpha(\Bim(X;\alpha))\subset\Bim(Y)$} is strongly Jordan.
Then the group $\Bim(X;\alpha)$ is Jordan.
In particular, if under these assumptions the map $\alpha$ is equivariant with respect
to the whole group $\Bim(X)$, then $\Bim(X)$ is Jordan.
\end{lemma}

\begin{proof}
As in the proof of Lemma~\ref{lemma:bounded-base}, suppose that the group $\Bim(X;\alpha)$ is not Jordan.
By Lemma~\ref{lemma:group-theory-2}, this means that the group $\Bim(X)_\alpha$ has unbounded finite subgroups.
In other words, the group~\mbox{$\Bim(X)_\alpha$} contains a countable family of subgroups
$G_i$, $i\in\mathbb{N}$, such
that the orders of~$G_i$ are unbounded.
On the other hand, by Lemma~\ref{lemma:fiberwise-embedding} there exists a constant $I$, a
typical fiber $F$ of the map $\alpha$, and an irreducible component $F'$ of $F$,
such that every $G_i$ contains a
subgroup $K_i$ of index at most $I$ which can be embedded into the group~\mbox{$\Bim(F')$}.
Since $\Bim(F')$ has bounded finite subgroups, we see that the orders of $K_i$ are bounded.
This contradicts the unboundedness of the orders of the groups~$G_i$.
\end{proof}

Lemma~\ref{lemma:bounded-fiber}
immediately implies

\begin{corollary}\label{corollary:bounded-fiber-finite}
Let $X$ and $Y$ be compact complex varieties, and let
$\alpha\colon X\dasharrow Y$ be a dominant meromorphic map
such that a typical fiber of $\alpha$ is finite.
Suppose that the group $\Bim(Y)$ is strongly Jordan.
Then the group $\Bim(X;\alpha)$ is Jordan.
In particular, if under these assumptions the map $\alpha$ is equivariant with respect
to the whole group $\Bim(X)$, then $\Bim(X)$ is Jordan.
\end{corollary}

\section{Pseudoautomorphisms}
\label{sect:pseudo}

In this section we consider groups of pseudoautomorphisms of compact K\"ahler manifolds, and establish the Jordan property
for groups of bimeromorphic selfmaps of three-dimensional compact K\"ahler varieties of Kodaira dimension zero.

Recall that a biholomorphic selfmap $f\colon X\dashrightarrow X$
of a compact complex variety $X$ is called a \emph{pseudoautomorphism}
if there are non-empty Zariski open subsets
$U_1,U_2\subset X$ such that~\mbox{$\mathrm{codim}_X(X\setminus U_i)\ge 2$},
and $f$ restricts to an isomorphism
$$
f\vert_{U_1}\colon U_1\xarr{\sim} U_2.
$$
The pseudoautomorphisms of $X$ form a subgroup in $\Bim(X)$
which we denote by $\operatorname{PAut}(X)$.

\begin{lemma}[{Negativity Lemma \cite[1.1]{Shokurov-1992-e-ba}, \cite[Lemma~1.3]{Wang2019}}]
\label{lemma:Negativity}
Let $f\colon \tilde V\to V$
be a proper bimeromorphic morphism between normal complex varieties.
Let $D$ be a Cartier divisor on $\tilde V$ such that $-D$ is $f$-nef. Then $D$
is effective if and only if $f_*D$ is effective.
\end{lemma}

The following assertion is well known to experts (see e.g.~\mbox{\cite[Lemma~4.3]{Kollar:flops}}).
We provide its proof for the convenience of the reader.

\begin{lemma}\label{lemma:pseudoaut}
Let $\chi\colon X \dashrightarrow X'$ be a bimeromorphic map of compact complex varieties with
 terminal singularities such that  $\upomega_{X'}$ is
nef. Then $\chi^{-1}$ does not contract any divisors.
 \end{lemma}

\begin{proof}
By definition of terminal singularities, for a sufficiently divisible positive integer $m$ the sheaves
$\upomega_X^{[m]}$ and $\upomega_{X'}^{[m]}$ are invertible.
Let
\[
\xymatrix@R=7pt{
&Y\ar[dl]_p\ar[dr]^q&
\\
X\ar@{-->}[rr]^{\chi}&&X'
}
\]
be   a common resolution, i.e. a commutative diagram where  $Y$
is a compact complex manifold, and~$p$ and $q$ are proper bimeromorphic
morphisms.
By making further blowups if necessary, we may assume that
the exceptional sets $\Exc(p)$ and $\Exc(q)$ are of pure codimension one.
Set~\mbox{$\Upsilon_p=p(\Exc(p))$} and~\mbox{$\Upsilon_q=q(\Exc(q))$}.
Then $p$ induces an isomorphism of open subsets~\mbox{$Y\setminus\Exc(p)$} and~\mbox{$X\setminus \Upsilon_p$}; similarly, $q$ induces an isomorphism of
$Y\setminus\Exc(q)$ and $X'\setminus \Upsilon_q$.
This implies that one can write
\begin{equation*}
\upomega_Y^{\otimes m}\cong
p^*\big(\upomega_X^{[m]}\big)\otimes
\OOO_Y\Big(m\Big(\sum a_i E_i +\sum b_j F_j\Big)\Big)
\cong q^*\big(\upomega_{X'}^{[m]}\big)
\otimes\OOO_Y\Big(m\Big(\sum c_iE_i +\sum d_k G_k\Big)\Big),\\
\end{equation*}
where $E_i$ (respectively, $F_j$, respectively, $G_k$) are exceptional
divisors with respect to
both $p$ and $q$ (respectively, $p$-exceptional
but not $q$-exceptional, respectively,  $q$-exceptional but not $p$-exceptional).
Thus, none of the three divisors  $\sum E_i$, $\sum  F_j$, and $\sum G_k$ has a common
irreducible component with any of the other two.
Since $X$ and $X'$ have only terminal singularities, the rational numbers
$a_i, \, b_j,\, c_i$, and $d_k$ are strictly positive.

Consider the Cartier divisor
\[
D=  m\Big(\sum (c_i-a_i) E_i+\sum d_k G_k-\sum b_j F_j\Big)
\]
Clearly, the push-forward $p_*D$ is effective. Since
\[
\OOO_Y(-D)\cong  q^* \upomega_{X'}^{[m]}\otimes p^*\upomega_X^{[-m]},
\]
we see that the divisor $-D$ is $p$-nef.
Thus
$D$ is effective  by Lemma~\ref{lemma:Negativity}.
Hence, one has $\sum  F_j=0$, i.e. the
map~$\chi^{-1}$ does not contract any divisors.
\end{proof}

\begin{corollary}\label{corollary:pseudoaut}
Let $X$ be a compact complex variety with
 terminal singularities and $\upomega_X$
nef. Then~\mbox{$\Bim(X)$} acts on $X$ by pseudo-automorphisms.
 \end{corollary}

The following result is due to A.\,Fujiki.

\begin{theorem}[{\cite[Corollary 3.3]{Fujiki:bimeromorphic}}]
\label{theorem:Fujiki-pseudo}
Let $X$ be a compact K\"ahler manifold, and let $g$ be its pseudoautomorphism.
Suppose that for a K\"ahler class $\alpha$ on $X$ its push-forward $g_*\alpha$ is again
a K\"ahler class. Then $g^{-1}$ is a morphism.
\end{theorem}

Theorem~\ref{theorem:Fujiki-pseudo} allows to study the group of pseudoautomorphisms of compact K\"ahler manifolds.

\begin{proposition}\label{proposition:pseudo}
Let $X$ be a compact K\"ahler manifold. Then the group of its pseudoautomorphisms is Jordan.
\end{proposition}

\begin{proof}
The group $\operatorname{PAut}(X)$ naturally acts on $H^2(X,\ZZ)$.
Let $\operatorname{PAut}'(X)$ be the kernel of this action. Then $\operatorname{PAut}'(X)$ preserves the K\"ahler class.
By Theorem~\ref{theorem:Fujiki-pseudo} the group $\operatorname{PAut}'(X)$ consists of biholomorphic automorphisms.
Hence the group $\operatorname{PAut}'(X)$
is Jordan by Theorem~\ref{theorem:Kim}.
On the other hand, the quotient
$\operatorname{PAut}(X)/\operatorname{PAut}'(X)$
acts faithfully on the finitely generated abelian group~\mbox{$H^2(X,\ZZ)$}, and thus has bounded finite subgroups
by Corollary~\ref{corollary:Minkowski}.
Therefore, the group~\mbox{$\operatorname{PAut}(X)$}
is Jordan by Lemma~\ref{lemma:group-theory}.
\end{proof}

\begin{remark}
In \cite{Fujiki:bimeromorphic}, Theorem~\ref{theorem:Fujiki-pseudo} is proved for manifolds (smooth varieties).
We do not know if this result, and thus Proposition~\ref{proposition:pseudo}, can be generalized to the case of
singular K\"ahler varieties.
\end{remark}

Applying Proposition~\ref{proposition:pseudo} together with Corollary~\ref{corollary:pseudoaut}, we obtain

\begin{corollary}\label{corollary:nef-pseudo}
Let $X$ be a compact K\"ahler manifold such that $\upomega_X$
is nef.
Then the group $\Bim(X)$ is Jordan.
 \end{corollary}

\section{Albanese map}
\label{section:Albanese}

In this section we collect information about Albanese maps of compact complex manifolds.

Let $X$ be a compact complex manifold.
We denote by
$$
\alb\colon X\xlongrightarrow{\hspace{1.8em}} \Alb(X)
$$
the Albanese morphism of $X$, see \cite[Definition~9.6]{Ueno1975}.
The map $\alb$ has the following universal property.
If $\zeta\colon X\to T$ is a morphism to an arbitrary complex
torus $T$, then there exists a unique homomorphism of complex tori~\mbox{$\xi\colon \Alb(X)\to T$} that fits into the commutative diagram
$$
\xymatrix{
X\ar@{->}[r]^\zeta\ar@{->}[d]_{\alb} & T\\
\Alb(X)\ar@{->}[ru]_\xi &
}	
$$

\begin{remark}
More generally, the Albanese morphism is well defined (and has the above universal property)
for compact complex varieties with rational singularities, see~\mbox{\cite[Theorem 3.3]{Graf2018}}
and~\mbox{\cite[Remark 3.4]{Graf2018}}.
\end{remark}

The next assertion is well known to experts,
but we provide its proof for the reader's convenience.
We will use it several times without further reference in the sequel.

\begin{proposition}
Let $X$ be a compact complex manifold.
There is a natural biregular action of the group $\Bim(X)$ on $\Alb(X)$ such that the
morphism $\alb$ is equivariant.
\end{proposition}

\begin{proof}
Let $\varphi\colon X\dashrightarrow X$ be an arbitrary bimeromorphic map.
The composition $\alb\circ\varphi$ is holomorphic by Proposition~\ref{proposition:torus-holomorphic}.
By the universal property of $\alb$
there exist a unique translation
$$
\mathrm{t}_a\colon \Alb(X) \xlongrightarrow{\hspace{1.8em}} \Alb(X)
$$
by an element $a\in \Alb(X)$ and
a unique homomorphism of complex tori
$$
\psi\colon \Alb(X) \xlongrightarrow{\hspace{1.8em}}\Alb(X)
$$
such that
\[
\alb\comp \varphi= \mathrm{t}_a\comp \psi \comp \alb.
\]
In other words, there exists a unique  morphism of abstract  varieties
$$
\theta_{\varphi}= \mathrm{t}_a\comp \psi\colon  \Alb(X) \xlongrightarrow{\hspace{1.8em}}\Alb(X)
$$
that
fits to the following commutative  diagram
\[
\xymatrix@C=47pt{
X\ar@{-->}[r]^{\varphi}\ar[d]^{\alb}&X\ar[d]^{\alb}
\\
\Alb(X)\ar[r]^{\theta_{\varphi}} & \Alb(X)
}
\]
It is easy to see that the correspondence is functorial: for any two
bimeromorphic maps~\mbox{$\varphi'\colon X\dashrightarrow X$}
and $\varphi''\colon X\dashrightarrow X$ one has
\[
 \theta_{\varphi'}\comp  \theta_{\varphi''}=\theta_{\varphi'\comp \varphi''}.
\]
In particular, the correspondence $\varphi \longmapsto \theta_{\varphi}$ defines
a group homomorphism
$$
\Bim(X) \xlongrightarrow{\hspace{1.8em}} \Aut(\Alb(X)).
$$
\end{proof}

\begin{remark}
\label{remark:q-1-connected-fibers}
If $X$ is a normal compact complex variety with rational singularities
and $\q(X)=1$,
then the fibers of the Albanese map $\alb\colon X\to A=\Alb(X)$ are connected.
Indeed, otherwise $\alb$ admits a non-trivial Stein factorisation
$$
\alb\colon X\xlongrightarrow{\hspace{0.7em}\alb'\hspace{0.7em}} A'\xlongrightarrow{\hspace{1.8em}} A.
$$
Let $\beta$ be the composition of $\alb'$ with the embedding of $A'$ into its Jacobian $J(A')$.
Then the map~\mbox{$\beta\colon X\to J(A')$} must factor through $\alb$ by the universal property of the Albanese map, which is
clearly impossible.
\end{remark}

Recall that a compact complex subvariety $Y$ of a complex torus $T$
is said to \emph{generate} $T$
if for some positive integer $n$
every point of $T$ can be represented as a sum of $n$ points of $Y$;
see~\mbox{\cite[Definition~9.13]{Ueno1975}}.
For instance, a proper subtorus of~$T$ (containing the neutral element of the group~$T$)
does \emph{not} generate~$T$.

\begin{proposition}\label{prop:subvariety-in-torus}
Let $T$ be a complex torus, and let $Y\subset T$ be a compact complex
subvariety.
Then there exists a canonically defined fibration $Y\to Z$ whose typical fiber is a subtorus~\mbox{$T_1\subset T$}
and the base $Z$ is a (possibly singular) compact complex variety
with $\dim(Z)=\kod(Z)$. Moreover,~$Z$ is a point if and only if
$Y$ is a translation of a subtorus $T_1\subset T$.
\end{proposition}

\begin{proof}
Take the largest subtorus of $T_1\subset T$ such that $Y$
is invariant under translations by all elements  of $T_1$, and then put  $Z=Y/T_1$.
See \cite[Theorem~10.9]{Ueno1975}
and its proof for details.
\end{proof}

\begin{lemma}
\label{lemma:Albanese}
Let $X$ be a three-dimensional compact complex
manifold with $\kod(X)\ge 0$.
Suppose that
$\alb$ is not surjective. Then the group
$\Bim(X)$ is Jordan.
\end{lemma}

\begin{proof}
Let $Y=\alb(X)\subsetneq\Alb(X)$.
By Proposition~\ref{prop:subvariety-in-torus}
there exists a canonically defined fibration
$$
\gamma\colon Y \xlongrightarrow{\hspace{1.8em}} Z,
$$
where $\dim(Z)=\kod(Z)$. The group $\Bim(Z)$ is finite
by Theorem~\ref{theorem:general-type}.
Recall that $Y$ generates~\mbox{$\Alb(X)$}, see \cite[Lemma~9.14]{Ueno1975}.
Since $Y$ contains the neutral element of the group~\mbox{$\Alb(X)$}, this implies that it is not contained in a proper subtorus of~\mbox{$\Alb(X)$}.
Hence $Z$ is not a point by Proposition~\ref{prop:subvariety-in-torus}.

Now we have a morphism
$$
\lambda=\gamma\circ\alb\colon X \xlongrightarrow{\hspace{1.8em}} Z
$$
that is equivariant with respect to
the group $\Bim(X)$.
Let $F$ be a typical fiber of $\lambda$, and let~$F'$ be its connected
component. Then $F'$ is a compact
complex manifold of dimension at most~$2$, and~\mbox{$\kod(F')\ge 0$} by Lemma~\ref{lemma:fiber-kappa}.
Thus the group $\Bim(F')$ is Jordan by Theorem~\ref{theorem:surface} and Remark~\ref{remark:curve}.
Therefore, the group $\Bim(X)$ is Jordan by
Lemma~\ref{lemma:bounded-base}.
\end{proof}

\begin{corollary}
\label{corollary:Albanese}
Let $X$ be a  three-dimensional compact complex manifold
with $\kod(X)\ge 0$.
Suppose that $\Bim(X)$ is not Jordan. Then
$\dim(\Alb(X))<\dim(X)$. In particular, if $X$ is K\"ahler, then~\mbox{$\q(X)<\dim(X)$}.
\end{corollary}

\begin{proof}
We know from Lemma~\ref{lemma:Albanese} that $\alb$ is surjective.
In particular, this implies that~\mbox{$\dim(\Alb(X))\le \dim(X)$}.
Suppose that $\dim(\Alb(X))=\dim(X)$. Then a typical fiber of~$\alb$
is finite. Applying Corollaries~\ref{corollary:bounded-fiber-finite} and~\ref{corollary:torus},
we see that the group $\Bim(X)$ is Jordan, which
is not the case by assumption.

If $X$ is a compact K\"ahler manifold, then~\mbox{$\dim(\Alb(X))=\q(X)$}.
\end{proof}

\section{Indeterminacy loci}
\label{section:indeterminacy}

In this section we make some observations concerning the indeterminacy loci of bimeromorphic maps
of three-dimensional compact complex varieties.

\begin{lemma}
\label{lemma:Ind-in-kf}
Let $X$ be a normal
three-dimensional compact complex variety,
and let $\chi\colon X \dashrightarrow X$ be a pseudoautomorphism.
Let~\mbox{$\beta\colon X\to Z$} be a morphism to a curve~$Z$
such that $\chi$ maps the fibers of~$\beta$ again to the fibers of~$\beta$.
Then the indeterminacy
locus  $\Ind(\chi)$ is contained in a finite set of fibers of~$\beta$.
\end{lemma}
\begin{proof}
Assume the contrary, i.e. there is an irreducible curve  $C\subset \Ind(\chi)$
that dominates $Z$.
Let~$\tilde Y$ be the normalization of the graph of $\chi$. Thus, there is a commutative diagram
\[
\xymatrix@R=7pt{
&\tilde {Y}\ar[dl]_p\ar[dr]^q&
\\
X\ar[dr]_\beta\ar@{-->}[rr]^{\chi}&&X\ar[dl]^\beta
\\
&Z&
}
\]
Since $\chi$ is not defined at a typical point of $C$, there exists a two-dimensional component $E$ of the $p$-exceptional set
that dominates $C$. Since the map $\tilde Y\to X\times X$ is finite onto its image,
no curves on $\tilde Y$ are contracted by both maps $p$ and $q$.
Hence $C'=q(E)$ is a curve, and
a typical fiber~$\Gamma$ of~\mbox{$q_E\colon E\to C'$} dominates $C$.
Thus $\Gamma$ meets the proper transform $\tilde F$ of any fiber~\mbox{$F=\beta^{-1}(z)$},
and so~$q(\tilde F)$ contains $C'$. In other words, $C'$ is contained in every fiber of $\beta$, which gives a contradiction.
\end{proof}

The next proposition is a relative analog of the well-known assertion about decomposition of certain maps into flops,
see for instance~\mbox{\cite[Theorem~4.9]{Kollar:flops}}.
The proof in our case follows the same scheme.
We outline it for convenience of the reader.

\begin{proposition}
\label{prop:pseudo-flips}
Let $f\colon X\to S$ and  $f'\colon X'\to S$ be proper morphisms with one-dimensional fibers,
where $X$ and $X'$ are three-dimensional K\"ahler varieties with terminal $\QQ$-factorial singularities and $S$ is a smooth surface.
Suppose that both $K_X$ and $K_{X'}$ are nef over~$S$.
Let $\chi\colon X\dasharrow X'$ be a bimeromorphic map such that it is an isomorphism in codimension one, and
the indeterminacy loci $\Ind(\chi)$ and $\Ind(\chi^{-1})$ are contained in (a finite set of) fibers
of~$f$ and~$f'$, respectively.
Suppose that there exists a commutative diagram
\[
\xymatrix{
X\ar[d]_{f}\ar@{-->}^{\chi}[rr]&&X'\ar[d]^{f'}
\\
S\ar@{-->}[rr]^{\zeta}&&S
}
\]
where $\zeta$ is a bimeromorphic map of $S$.
Then $\zeta$ is an isomorphism, and $\chi$ is a composition of flops in the curves that are contained in (a finite set of) fibers over $S$.
\end{proposition}

\begin{proof}[Outline of the proof]
By assumption, $\chi$ is an automorphism outside a finite union of fibers of $f$.
This means that $\zeta$ induces an automorphism of a smooth open subset $S^o\subset S$
that is a complement to a finite subset of $S$. Therefore, by Hartogs's extension theorem $\zeta$ is
an isomorphism on the whole~$S$.

To prove that $\chi$ is a composition of flops in the curves contracted by $f$,
we may assume that $S$ is a small analytic neighborhood of a point  $s\in S$.
Thus, we may assume that $f'$ is a projective morphism, because it has one-dimensional fibers.
In other words, there exists a divisor $D'$ on $X'$ that is very ample over $S$:
one can construct $D'$ as a union of discs meeting the fiber $f'^{-1}(s)$ transversely.
Similarly, we see that the morphism $f$ is projective.

Let $D\subset X$ be the proper transform of $D'$.
For small positive $\varepsilon$, run a $(K_X+\varepsilon D)$-MMP over~$S$.
This is possible because $f$ is projective, see~\mbox{\cite[\S4]{Nakayama:semi}}.
Every step of this MMP is a flop, and we end up with~$X'$.
\end{proof}

\begin{corollary}
\label{corollary:Kollar}
In the notation of Proposition \xref{prop:pseudo-flips}, set~\mbox{$\Theta=f(\Sing(X))$}.
Then~\mbox{$\zeta(\Theta)=\Theta$}.
\end{corollary}

\begin{proof}
A flop preserves the analytic types of singularities of the ambient variety, see~\mbox{\cite[Theorem~2.4]{Kollar:flops}}.
Thus, the required assertion follows from Proposition~\ref{prop:pseudo-flips}.
\end{proof}

The following fact is well-known.

\begin{lemma}[{see e.g. \cite[Lemma~1.8]{Kulikov:codim2}}]
\label{lemma:ind-new}
Let $X$ be a three-dimensional compact
complex variety with at worst terminal singularities,
and let $\chi\colon X \dashrightarrow X$ be a pseudoautomorphism.
Then any irreducible component of the indeterminacy
locus  $\Ind(\chi)$ is a rational curve.
\end{lemma}

\section{More on compact complex surfaces}
\label{section:more}

In this section we make some additional observations on compact complex surfaces and their automorphism groups.

\begin{lemma}
\label{lemma:chi=0}
Let $S$ be a smooth compact K\"ahler surface
with $0\le \kod(S)\le 1$ and $\q(S)>0$.
Then~\mbox{$\chit(S)=0$} if and only if $S$ contains no rational curves.
\end{lemma}

\begin{proof}
We may assume that $S$ is minimal.
Indeed, according to the Enriques--Kodaira classification (see e.g.~\mbox{\cite[\S\,VI.1]{BHPV}}),
every minimal surface of non-negative Kodaira dimension has a non-negative
topological Euler characteristic. Thus, any non-minimal compact complex surface
of non-negative Kodaira dimension has a positive topological Euler characteristic
(and contains rational curves).

Suppose that $\kod(S)=0$. Since $\q(S)>0$, we conclude that $S$ is either
a complex torus, or a bielliptic surface.
In both of these cases $\chit(S)=0$ and $S$ contains no rational curves.

Suppose that $\kod(S)=1$.  Let $\kf\colon S\to B$ be the pluricanonical map; thus, $\kf$ is an elliptic fibration
over a curve $B$.
If the genus $\g(B)>0$, then all rational curves on $S$ are contained in a finite set of fibers of~$\kf$,
and $\chit(S)=\sum\chit(F_i)$, where $F_i$ are degenerate fibers.
Clearly, $\chit(S)>0$ if and only if~$\kf$ has a fiber that is not an elliptic curve.
Such a fiber is a union of rational curves.

Assume that $B\cong \PP^1$. We are going to show that in this case $\chit(S)=0$ and $S$ does not contain rational curves.
Indeed, the Albanese map $\alb\colon S\to A$ cannot be constant on the fibers of~$\kf$, because otherwise it would factor through
the Albanese map of~$\PP^1$, which is impossible since~\mbox{$\q(S)>0$}.
Therefore, any fiber $S_b$, $b\in B$, of the map $\kf$ does not contain rational curves; in other words, it is of type ${}_m\mathrm{I}_0$
in the Kodaira classification of degenerate fibers of elliptic fibrations (see for instance~\mbox{\cite[\S\,V.7]{BHPV}}).
Hence, one has
$$
\chit(S)=\sum\limits_{b\in B}\chit(S_b)=0.
$$

We claim that $\q(S)=1$. Indeed, for a typical fiber $S_o$ of $\kf$
the image $\alb(S_o)$ is an elliptic curve on $A$. Consider the quotient map
$A\to A'=A/\alb(S_o)$. The image of $\alb(S_o)$ is a point on $A'$, and so for any fiber $S_b$ of $\kf$ the image
of $\alb(S_b)$ on $A'$ is also a point. Thus, the image of $\alb(S)$ on $A'$ is a curve dominated by the
(rational) base of the fibration $\kf$. Since $A'$ contains no rational curves, we conclude that the image of
$\alb(S)$ on $A'$ is a point.
Hence $\alb(S)=\alb(S_o)=A$.

Therefore, $A$ is an elliptic curve
and $\alb$ has connected fibers, see Remark~\ref{remark:q-1-connected-fibers}.
Since $S$ is a K\"ahler surface,
we know that $\rb_1(S)=\rb_3(S)=2\q(S)$ and
\[
\rb_2(S)=\chit(S)-2+4\q(S)=2.
\]
Hence all the fibers of $\alb$ are irreducible.

Let $F_1,\dots, F_r$ be all the singular  fibers of $\alb$,
let $m_1,\dots, m_r$ be their multiplicities, and let $F$ be a typical fiber of $\alb$.
Thus, $F$ is numerically equivalent to $m_iF_i$, and $\chit(F)=2-2\g(F)=-K_S\cdot F$.
Similarly, one has
$$
\chit(F_i)\ge 2-2\p(F_i)=-K_S\cdot F_i=-\frac{1}{m_i}K_S\cdot F.
$$
Thus we have
\[
 0=\chit(S)=\sum\limits_{i}\big(\chit(F_i)-\chit(F)\big)\ge
 K_S\cdot F \cdot \sum\limits_{i} \left(1-\frac 1{m_i}\right).
\]
Since $S$ is not covered by rational curves, we know that $K_S\cdot F>0$. This shows that $m_i=1$ and~\mbox{$\chit(F_i)= 2-2\p(F_i)$} for all~$i$.
Therefore, $\alb$ is a smooth morphism.
Now if $S$ contains a rational curve $C$, then $C$
must be a (smooth) fiber of $\alb$. But then
$$
-2=2\g(C)-2=K_S\cdot C=K_S\cdot F\ge 0,
$$
which is a contradiction.
\end{proof}

The next result was proved in \cite{Prokhorov-Shramov-BFS}.

\begin{proposition}[{\cite[Corollary~4.1]{Prokhorov-Shramov-BFS}}]
\label{proposition:BFS-for-Kahler-surfaces}
Let $S$ be a compact K\"ahler surface with~\mbox{$\kod(S)\ge 0$}.
Suppose that the group $\Bim(S)$
has unbounded finite subgroups. Then either $\kod(S)=1$, or $S$ is bimeromorphic to a complex torus or a bielliptic surface.
\end{proposition}

The following fact is a version of Proposition~\ref{proposition:BFS-for-Kahler-surfaces} for automorphism groups.

\begin{proposition}
\label{prop:bfs}
Let $S$ be a minimal compact K\"ahler
surface with $\kod(S)\ge 0$.
Suppose that the group $\Aut(S)$ has unbounded finite subgroups.
Then $\q(S)>0$ and $\chit(S)=0$. Moreover,
$S$ is either a  complex torus,  or a bielliptic surface, or a surface with~\mbox{$\kod(S)=1$}.
\end{proposition}

\begin{proof}
One has $\kod(S)<2$ by Theorem~\ref{theorem:general-type}.
If $\kod(S)= 0$, then by Proposition~\ref{proposition:BFS-for-Kahler-surfaces}
the surface $S$ is either a complex torus,
or a bielliptic surface.
In both of these cases the required assertions clearly hold.

Suppose that $\kod(S)=1$. Let $\kf\colon S\to B$ be the pluricanonical fibration.
If $\kf$ has a fiber~\mbox{$S_b=\kf^{-1}(b)$} over some point $b\in B$ which is not of type ${}_m\mathrm{I}_0$,
then a finite index subgroup $\Gamma\subset \Aut(S)$ fixes a singular point
of $F_\red$. Since $S$ is K\"ahler, this implies that $\Aut(S)$ has bounded finite subgroups
by Theorem~\ref{theorem:fixed-point-BFS}.
Thus we may assume that all the fibers of $\kf$ are of type ${}_m\mathrm{I}_0$.
Hence, one has~\mbox{$\chit(S)=0$}. By Noether formula this gives $\chi(\OOO_S)=0$, so that~\mbox{$\q(S)=1+\pg(S)>0$}.
\end{proof}

\section{Proof of the main theorem}
\label{section:proof}

In this section we complete the proof of Theorem~\ref{theorem:main}.

\begin{lemma}
\label{lemma1}
Let $X$ be a three-dimensional compact K\"ahler variety with $\kod(X)\ge 0$.
Suppose that there exists a dominant $\Bim(X)$-equivariant meromorphic map~\mbox{$f\colon X  \dashrightarrow Z$} to a smooth projective curve $Z$ of positive genus.
Then $\Bim(X)$ is Jordan.
\end{lemma}

\begin{proof}
Running the MMP on $X$, we  may assume that $X$ has at worst terminal $\QQ$-factorial singularities and $\upomega_X$ is nef.
Since $Z$ is not a rational curve,
$f$ is holomorphic: otherwise its composition with
the embedding of $Z$ into its Jacobian is a non-holomorphic
map to a complex torus, which is impossible by
Proposition~\ref{proposition:torus-holomorphic}.
Applying the Stein factorization, we may assume that
the fibers of~$f$ are connected.
Thus, any smooth fiber $F$ of $f$ is a minimal surface of non-negative Kodaira dimension; in particular,
one has $\Bim(F)=\Aut(F)$.
Note also that $\Bim(X)$ acts on $Z$ by automorphisms.

Suppose that $\Bim(X)$ is not Jordan.
Then it follows from Lemma~\ref{lemma:bounded-fiber} and Remark~\ref{remark:curve} that for a typical fiber~$F$ of~$f$
the group $\Aut(F)$ has unbounded finite subgroups.
Therefore, by Proposition~\ref{prop:bfs} one has $0\le \kod(F)\le 1$, $\chit(F)=0$, and $\q(F)>0$.
Recall that the topological Euler characteristic is constant in smooth families of compact manifolds.
Hence for \emph{any} smooth fiber~$F$ of~$f$
one has $\chit(F)=0$. Furthermore, since the irregularity of a compact complex surface
is uniquely determined by the first Betti number (see e.g.~\mbox{\cite[Theorem~IV.2.7]{BHPV}}),
it is also constant in smooth
families, so that $\q(F)>0$. Also, we see that $\kod(F)\ge 0$ by the adjunction formula, and~\mbox{$K_F^2=0$}, so that~\mbox{$\kod(F)\le 1$}.
This means that any smooth fiber of $f$ contains no rational curves
by Lemma~\ref{lemma:chi=0}.

Let $F$ be a smooth fiber of $f$, and let $\gamma\in \Bim(X)$ be an arbitrary element.
By Lemma~\ref{lemma:ind-new} any  irreducible component  $C$ of the indeterminacy
locus  $\Ind(\gamma)$ is a rational curve. Since $\g(Z)>0$, the curve $C$ cannot dominate $Z$.
On the other hand, we already know that $C\not\subset F$.
Hence $C$ is disjoint from $F$, i.e. $\gamma$ is holomorphic near $F$.
Then $\gamma(F)$ also does not contain rational curves and
so $\gamma^{-1}$ is holomorphic near $\gamma(F)$.
This implies that $\gamma$ is an isomorphism near $F$.
Hence $\gamma(F)$ is a smooth fiber, i.e. any element $\gamma\in \Bim(X)$
maps smooth fibers to smooth fibers (and vice versa, maps singular fibers to singular fibers).

Now let $F_1,\ldots,F_r$, $r\ge 0$, be all the singular fibers of $f$.
Then $\Bim(X)$ acts biholomorphically on the complement
$X\setminus \bigcup F_i$.
If $r=0$, then $\Bim(X)$ acts on the whole $X$ by
automorphisms.
In this case~\mbox{$\Bim(X)=\Aut(X)$} is Jordan by Theorem~\ref{theorem:Kim}.
If $r>0$, then $\Bim(X)$ permutes the points~\mbox{$f(F_1),\ldots,f(F_r)$}.
Therefore, $\Bim(X)$ is Jordan by
Corollary~\ref{corollary:3-fold-over-elliptic-curve-fixed-point}.
\end{proof}

\begin{lemma}
\label{lemma2}
Let $X$ be a three-dimensional compact K\"ahler variety with $\kod(X)\ge 0$.
Suppose that there exists a dominant $\Bim(X)$-equivariant meromorphic map~\mbox{$f\colon X  \dashrightarrow Z$} to a compact complex surface $Z$ with $\kod(Z)\ge 0$.
Then $\Bim(X)$ is Jordan.
\end{lemma}

\begin{proof}
Running the MMP, we may assume that the singularities of $X$ are at worst terminal $\QQ$-factorial and $\upomega_X$ is nef.
Then \mbox{$\Bim(X)$} acts on $X$ by pseudo-automorphisms, see Corollary~\ref{corollary:pseudoaut}.
Furthermore, we may assume that $Z$ is a minimal surface.
So, the induced action of~\mbox{$\Bim(X)$} on~$Z$ is biholomorphic.
We may also assume that the group $\Aut(Z)$ has unbounded finite subgroups, because otherwise $\Bim(X)$ is Jordan by
Corollary~\ref{corollary:3-fold-over-surface}.
Note that the surface~$Z$ is K\"ahler by~\mbox{\cite[Theorem~5]{Varouchas}}. Thus, we know from Proposition~\ref{prop:bfs} that
$Z$ is ether a complex torus, or a bielliptic surface, or a surface with $\kod(Z)=1$.

Assume that $\kod(Z)=1$. Let $\psi\colon X\dashrightarrow B$ be the composition of $f$ with the pluricanonical fibration $Z\to B$.
The image of $\Bim(X)$ in $\Aut(B)$
is finite by \cite[Proposition~1.2]{Prokhorov-Shramov-BFS}. Hence~$\Bim(X)$ is Jordan by Corollary~\ref{corollary:3-fold-over-curve}.
The same argument works if $Z$ is a bielliptic surface:
in this case there is an $\Aut(Z)$-equivariant elliptic fibration $Z\to \PP^1$
(see~\mbox{\cite[\S\,V.5]{BHPV}}), and by Kodaira's canonical bundle formula (see~\mbox{\cite[Theorem~V.12.1]{BHPV}})
it has~$3$ or~$4$ multiple fibers.
Hence the image of $\Bim(X)$ in $\Aut(\PP^1)$
is finite and $\Bim(X)$ is again Jordan.

Finally, assume that $Z$ is a complex torus and $\Bim(X)$ is not Jordan.
The map $f$ is holomorphic in this case, see Proposition~\ref{proposition:torus-holomorphic}.

Suppose that $f$ has a two-dimensional fiber.
Denote by $\Sigma$ the image under $f$ of the union of all
two-dimensional fibers of $f$. Then $\Sigma$ is a finite non-empty subset of $Z$.
Since a pseudoautomorphism cannot contract a two-dimensional fiber, we conclude that
$\Sigma$ is invariant under the action of the image of $\Bim(X)$ on $Z$.
Thus, the group $\Bim(X)$ is Jordan by Corollary~\ref{corollary:over-torus}.

Now suppose that all fibers of $f$ are one-dimensional. By
Lemma~\ref{lemma:Ind-in-kf},
the indeterminacy locus of every bimeromorphic selfmap of $X$ is contained
in a finite set of fibers of $f$. Therefore, it follows from Corollary~\ref{corollary:Kollar} that
the group $\Bim(X)$ preserves the image of the singular locus of $X$ on~$Z$. Thus,
$\Bim(X)$ preserves the image of $\Sing(X)$ on $Z$, which is a finite non-empty subset of $Z$ by assumption.
Again by Corollary~\ref{corollary:over-torus},
this implies that the group $\Bim(X)$ is Jordan.
\end{proof}

Now we are ready to prove Theorem~\ref{theorem:main}.

\begin{proof}[Proof of Theorem~\ref{theorem:main}]
Let $\alb\colon X\to A$ be the Albanese map. By Corollary~\ref{corollary:Albanese} we may assume that $\q(X)\le 2$, and by Lemma~\ref{lemma:Albanese} we may assume that~$\alb$ is surjective.
Now the assertion of the theorem is given by Lemma~\ref{lemma1} if $\q(X)=1$, and by Lemma~\ref{lemma2} if $\q(X)=2$.
\end{proof}


\def\cprime{$'$}

\end{document}